\documentclass[12pt,letterpaper]{amsart}
\makeatletter
\@namedef{subjclassname@2020}{%
\textup{2020} Mathematics Subject Classification}
\makeatother
\usepackage{geometry,xcolor}
\usepackage{setspace}
\usepackage{mathpazo} 
\usepackage[colorlinks,allcolors=blue]{hyperref} 
\usepackage{xpatch,multirow}
\xpatchcmd{\proof}{\itshape}{\bfseries}{}{}
\newtheorem{theorem}{Theorem}

\newtheorem{corollary}{Corollary}
\newtheorem{lemma}{Lemma}

\theoremstyle{remark}
\newtheorem{remark}{Remark}
\newtheorem{definition}{Definition}

\usepackage{soul}
\allowdisplaybreaks

\title{The K\"ahler submanifolds between the ball bundles and the complex Euclidean space}

\author{Mingming Chen}
\address[Mingming Chen]{School of Mathematics and Statistics, Henan Normal University, Xinxiang 453007, China}
\email{chenmingming105@126.com}

\author{Yihong Hao}
\address[Yihong Hao]{Department of Mathematics, Northwest University, Xi'an \rm{710127}, China}
\email{haoyihong@126.com}

\author{An Wang}
\address[An Wang]{School of Mathematical Sciences, Capital Normal University, Beijing \rm{100048}, China}
\email{wangan@cnu.edu.cn}

\subjclass[2010]{32H02, 32Q40,  53B35}
\keywords{K\"ahler submanifold, Isometric embedding, Nash algebraic function, Ball bundle}
\date{\today}
\begin{document}
	
\begin{abstract}
In this paper, we provide a sufficient condition  on the non-existence of the common K\"ahler submanifolds between
the complex Euclidean space and the ball bundles of some Hermitian vector bundles over  K\"ahler manifolds.
  Then we get  the non-existence theorems on several classes of  ball bundles whose base spaces are  Hermitian symmetric spaces
  or the complete K\"ahler-Einstein manifolds.
\end{abstract}
\maketitle

\section{Introduction}\label{sec:1}
In 1953, Calabi \cite{Calabi1953} studied the existence of the holomorphic isometric embedding between
 K\"ahler manifolds and  complex space forms.
Following his researches,  many other important works have appeared on the characterization and classification of K\"ahler submanifolds of complex space forms \cite{Di-Loi,Di2012,Hao-Wang-Zhang,Loi2018,Loi2023,Loi-Zedda}, as well as the study of related number theory \cite{Clozel2003,Mok2011,Mok2012}.
 It is easy to see that the property of two K\"ahler manifolds
sharing a common K\"ahler submanifold is beneficial to discussing the existence
of the holomorphic isometric embedding between them.
Now,  if two K\"ahler manifolds  satisfy this property, they   are called
relatives by Di Scala and Loi \cite{Di2010}. Otherwise, they are not relatives.
Actually, the study of the relativity problem between two K\"ahler manifolds
 dates back to
Umehara \cite{Um1987} who proved that two complex space forms with Einstein constants of different signs
cannot share a common K\"ahler submanifold with induced metrics.
In  \cite{Di2010}, Di Scala and Loi proved that any  complex bounded domain  with its Bergman metric
and a projective K\"ahler manifold  are not relatives.
After that, Mossa showed  that a bounded homogeneous domain with a homogeneous K\"ahler metric
can not be a relative to a projective K\"ahler manifold \cite{Mossa2013}.
For the  relativity problem on indefinite complex space forms, the readers are referred  to \cite{Cheng2017}\cite{Cheng2023}\cite{Um1988}\cite{Zhang2023} and  a survey of some recent studies \cite{Yuan2019}.

 Since any irreducible Hermitian symmetric space of compact type can be holomorphically isometrically embedded into a complex projective space by the classical Nakagawa-Takagi embedding, it follows from the
result of Umehara \cite{Um1987}, the complex Euclidean space and the irreducible Hermitian symmetric
space of compact type cannot be relatives.
 Later, Huang and Yuan showed that a complex Euclidean space and a
Hermitian symmetric space of noncompact type cannot be relatives \cite{Huang2015}.
They introduced  Nash algebraic functions  as their powerful tool to  study  the existence of common complex submanifolds.
It has been demonstrated that Hermitian symmetric spaces of noncompact type \cite{Huang2015},  symmetrized polydisk \cite{Su2018},
Cartan-Hartogs domains \cite{ChengN2017,Zhang2022}, bounded homogeneous domains \cite{Cheng2021}, minimal domains \cite{Cheng2021},
and some Hua domains \cite{Ma2024} are not relatives to the complex Euclidean space.
In retrospect,  these studies on  the relativity problem between an arbitrary K\"ahler manifold and the complex Euclidean space,  all canonical metrics are Bergman metrics. The main reason is  the explicit forms of the Bergman kernel on such domains had been obtained.

Notice that the definition of the Bergman metric on bounded domains had been generalized to complex manifolds. And any pseudoconvex Hartogs domain over a bounded domain with fiber dimensional $k$ can be viewed as a special ball bundle in a  trivial Hermitian vector bundle $E$ of rank $k$.
For an arbitrary Hermitian vector bundle $E$, denote by $(\mathbb{B}(E), g_{\mathbb{B}(E)})$ the ball bundle equipped with its canonical metric (the complete Bergman metric or the complete K\"ahler-Einstein metric), it is  natural to ask whether the common submanifold between the complex Euclidean space and $(\mathbb{B}(E), g_{\mathbb{B}(E)})$ exists or not?
We introduce the exponent Nash-algebraic K\"ahler manifold in  Definition \ref{def EN} and  prove the following  result which will help us to solve the question.
\begin{theorem}\label{thm1}
Let  $\pi:(L, h)\rightarrow M$ be a Hermitian line bundle over
 a  real-analytic K\"ahler manifold $(M, g_{M})$ such that the K\"ahler form of $g_{M}$ satisfying $\sqrt{-1}\partial \bar\partial\log h=l\omega_{M}$ for a real number $l$.
 Let $E_{k}$ be the $k$-th direct sum of  $(L, h)$ for $k\in \mathbb{Z}^{+}$.
The ansatz
\begin{equation}\label{equ:om}
\omega_{\mathbb{B}(E_{k})}=\alpha\pi^{*}(\mathrm{Ric}(\omega_{M}))+\beta\pi^{*}(\omega_M)+\sqrt{-1}\partial\bar\partial u(\|v\|^{2}_{H_{k}})
\end{equation}
is an $(1, 1)$-form on the total space of $E_{k}$, where $u$ is  a real-valued smooth function, $||v||_{H_{k}}^{2}=\langle\xi,\xi\rangle h(z)=\sum_{j=1}^{k}\xi_{j}\bar\xi_{j}h(z)$, $\alpha, \beta\in \mathbb{R}$.
The ball bundle is defined by
$
 \mathbb{B}(E_{k}) := \{v \in E_{k} : ||v||_{H_{k}}^{2} < 1\}.
$

Let $V$ be a connected open subset in $\mathbb{C}$. Suppose that   $F:V\to \mathbb{C}^{n}$ and  $G:V\to \mathbb{B}(E_{k})$  are   holomorphic mappings such that
\begin{equation}\label{01}
 	F^{*} \omega _{\mathbb{C}^{n}  } =\mu G^* \omega_{\mathbb{B}(E_{k})} ~~on ~~V
\end{equation}
for a real constant $\mu$.
If $(M, g_{M})$ is an exponent Nash-algebraic K\"ahler manifold and $\exp u(x)$ is
rational, then $F$ must be a constant map.
\end{theorem}

In order to  state  Theorem \ref{thm1} clearly, we induce two conditions:
\begin{description}
  \item[A] K\"ahler manifold  $(M, g_{M})$  is  exponent Nash-algebraic;
  \item[B] The real function $\exp u(x):[0,1)\rightarrow \mathbb{R}^{+}$ \ is \ rational \ in  \ $x$.
\end{description}

The result of  Theorem \ref{thm1} may not be true if we remove condition A or B. In fact, assume that the ansatz
$
\omega_{\mathbb{B}(E_{k})}
$
in \eqref{equ:om} induces a K\"ahler metric $g_{\mathbb{B}(E_{k})}$ on  $\mathbb{B}(E_{k})$.
Then, the $(1,1)$ forms
\begin{equation}\label{two metrics}
\omega'_{M}:=\alpha\mathrm{Ric}(\omega_{M})+\beta\omega_M \ \ \text{and} \ \
\omega_{D_{p}}:=\sqrt{-1}\partial\bar\partial u(||v||_{H_{k}}^{2})|_{D_p}
\end{equation}
  induce  K\"ahler metrics $g'_{M}$ on $M$ and  $g_{D_{p}}$ on the domain
 $D_{p}=\{v \in E_{k}|_{p}: ||v||_{H_{k}} < 1\}$ respectively.
It is easy to see that  $(M, g'_M)$ and $(D_{p}, g_{D_{p}})$ are the K\"ahler submanifolds of $(\mathbb{B}(E_{k}), g_{\mathbb{B}(E_{k})})$.
Obviously, $(\mathbb{B}(E_{k}), g_{\mathbb{B}(E_{k})})$ and $(\mathbb{C}^{n}, g_{\mathbb{C}^{n}})$ have common K\"ahler submanifolds if
$(M, g'_M)$  or  $(D_{p}, g_{D_{p}})$  is so.
However, condition $\mathbf{A}$ can rule out the existence of the common K\"ahler submanfold between $(M, g_M')$  and $(\mathbb{C}^{n}, g_{\mathbb{C}^{n}})$ by Lemma \ref{prop1.3}.
 It was proved by Cheng and the second author \cite{Cheng2021}  that
 for any positive real number $\mu$, K\"ahler manifold $(M, \mu g_{M})$ and
 the complex Euclidean space  do not have common K\"ahler submanifolds  if
 $g_{M}$ is exponent Nash-algebraic. Notice that  the metric  $g_{D_{p}}$ in \eqref{two metrics}  is
exponent Nash-algebraic on $D_{p}$ if $\exp u$ is rational. Hence, condition $\mathbf{B}$ also rules out the existence of the common K\"ahler submanfold between $(D_{p}, g_{D_{p}})$ and $(\mathbb{C}^{n}, g_{\mathbb{C}^{n}})$.

Suppose that  $(\mathbb{B}(E_{k}), g_{\mathbb{B}(E_{k})})$ is a K\"ahler manifold with K\"ahler form $\omega_{\mathbb{B}(E_{k})}$ and  there exists  common K\"ahler submanifolds between it and  $(\mathbb{C}^{n}, g_{\mathbb{C}^{n}})$. Then the equation  \eqref{01} follows from it.
As an application of Theorem \ref{thm1}, we study the non-existence of the common submanifold between the complex Euclidean space and  several classes of  ball bundles equipped with their canonical metrics.
For the Bergman metric, we consider three kinds of  ball bundles over Hermitian symmetric spaces.
The first  non-existence result is given by Corollary \ref{cor1} about the ball bundles in the direct sum bundle of the top exterior product $\bigwedge^{n} T(1,0)$ of the holomorphic tangent bundle over the compact Hermitian symmetric space.
The second  non-existence result is given by Corollary \ref{cor2} about the Hartogs domains over bounded symmetric domains.
The last one is  an existence result  given by Corollary \ref{cor3} about the Hartogs domains over the complete flat spaces.
For K\"ahler-Einstein metric, we consider a class of ball bundles in the direct sum bundles of the negative Hermitian line bundles over the complete K\"ahler-Einstein manifolds with negative Ricci curvatures and get  the non-existence result in Corollary \ref{cor4}.

The organization of this article is as follows. We start in Sect.\ref{sec2}
by presenting  some
properties of Nash function.
In Sect.\ref{sec3}, we introduce the ansatz proposed  by Ebenfelt, Xiao and Xu on  ball bundles and give the proof
of Theorem \ref{thm1}. As an application, we obtain Corollary  \ref{cor1}-\ref{cor4} in the last section.

\section{Some  properties of Nash algebraic function}\label{sec2}
Let $D\subset{\mathbb C}^n$ be an open subset and $f$ be a holomorphic function on $D$.
We say that $f$ is a Nash function at $z_{0}\in D$ if there exists an open neighbourhood $U$ of $z_{0}$
and a nonzero polynomial $ P : \mathbb{C}^n \times \mathbb{C}\rightarrow \mathbb{C}$,
such that $P(z, f (z)) \equiv0$ for $z \in U$.
Thus, a holomorphic function defined on $D$ is called  a Nash function if it is a Nash function at every point in $D$.
We denote the family of Nash functions on $D$  by $\mathcal{N}(D)$.
Some basic properties of Nash functions are collected by the following lemmas.
For more details, we refer the readers to  \cite{Huang1994,Huang2014,Tw1990}.
\begin{lemma}\cite{Tw1990}\label{Tw1990}
Let $f, g\in \mathcal{N}(D)$. Then the following holds:
\begin{description}
    \item[(1)] $f\pm g$, $fg$, $\displaystyle\frac{f}{g}\in \mathcal{N}(D)$ and
               $\displaystyle\frac{\partial f}{\partial z_{i}} \in \mathcal{N}(D)$ for each $1\leq i\leq n.$
    \item[(2)] For any fixed $(z_{k+1}^{0},\cdots, z_{n}^{0})$,
               $f(z_{1},\cdots, z_{k}, z_{k+1}^{0},\cdots, z_{n}^{0})$
               is a Nash function in $z_{1}$,$\cdots$, $z_{k}$.
    \item[(3)] Let $D_{1}$ and $D_{2}$ be two open subsets of ${\mathbb C}^n$.
               Suppose that $\varphi(z, w)\in \mathrm{Hol}(D_{1}\times D_{2} )$ and $\varphi(z, w_{0})\in \mathcal{N}(D_{1})$
               for any fixed $w_{0}\in D_{2}$ (respectively, $\varphi(z_{0}, w)\in \mathcal{N}(D_{2})$,
               for any fixed $z_{0}\in D_{1}$). Then $\varphi(z, w)\in \mathcal{N}(D_{1}\times D_{2} )$.
    \item[(4)] The composition of Nash functions is a Nash function.
  \end{description}
\end{lemma}
\begin{lemma}\label{Ma2024}\cite{Ma2024}
Let $V\subset \mathbb{C} ^{\kappa }$
be a connected open set,
$\xi =\left ( \xi _{1} ,\cdots ,\xi _{\kappa }  \right )\in V$,
and $H_{1} \left ( \xi   \right )$, $\cdots ,H_{\kappa _{1}+\kappa _{2} } \left ( \xi \right ) $, $H\left ( \xi  \right )$
be Nash functions on $V$,
$\mu,\mu _{1},\cdots ,\mu _{\kappa _{1} +\kappa _{2}}\in \mathbb{R}\setminus \left \{ 0 \right \}$
and $R$ be a multi-variable holomorphic rational function.
Assume that
\begin{equation}\label{E2.1}
\exp \left\{H^{\mu}(\xi)\right\}
=R\left(H_{1}^{\mu_{1}}(\xi), \cdots, H_{\kappa_{1}}^{\mu_{\kappa_{1}}}(\xi)\right)
\prod_{j=\kappa_{1}+1}^{\kappa_{1}+\kappa_{2}} H_{j}^{\mu_{j}}(\xi)  ,	
\end{equation}
then $H\left ( \xi  \right )$ must be a  constant on $V$.
\end{lemma}

The proof of this lemma  in \cite{Ma2024}  shows that the condition `` $\mu _{1},\cdots ,\mu _{\kappa _{1} +\kappa _{2}}\in \mathbb{R}\setminus \left \{ 0 \right \}$ '' can be replaced by
``$\mu _{1},\cdots ,\mu _{\kappa _{1} +\kappa _{2}}\in \mathbb{R}$ and at least one of  them is non zero''.
The contradictory equation \eqref{E2.1} for  Nash power functions is the key in the proof of Theorem \ref{thm1}.
The following result was originally contained in  \cite{Huang2015} and was rewritten as a lemma in \cite{Cheng2023}.
\begin{lemma}\label{field}
\cite{Huang2015} \cite{Cheng2023}
Let $U$ be an open set in  $\mathbb{C}$, $0\in U$. Let $$S=\{h_1, \cdots, h_l\}:=\{f_1, \cdots, f_n, g_1, \cdots, g_m, g_{m+1}\}$$ be a set of holomorphic functions on $U$. Denote by $\Re$ and $\mathfrak{F} = \Re \left ( S \right ) $
the field of rational functions and the field extension over rational function on $U$ respectively. Let $r$ be the transcendence
degree of the field  $\mathfrak{F}/\Re $ and
$\left \{h _{1},\cdots ,h _{r} \right \} \subset S$ be a maximal algebraic independent subset over $\Re$.

$\bullet$ If $r=0$,  then all elements in S are Nash algebraic.

$\bullet$ If $r>0$, there are $\left \{h _{1},\cdots ,h _{r} \right \} \subset S$ be a maximal algebraic independent subset over $\Re$, thus $\left \{h_j, h _{1},\cdots ,h _{r} \right \} \subset S$ is algebraic dependent over $\Re$. By the  existence and uniqueness of the implicit function theorem, there exists $D\subset U$ with $0\in \overline{D}$ and holomorphic Nash algebraic
functions $\hat{h } _{j} \left ( z,X \right )$
defined in a neighborhood $\hat{D}$ of
$\left \{ \left ( z,\Phi \left ( z \right ) |z\in D \right )  \right \} \subset \mathbb{C}^{n} \times \mathbb{C}^{r} $,
such that
$$h_{j}(z)=\hat{h}_{j}(z, \Phi(z)), \quad \forall j=1, \cdots, l,$$
where  $\Phi(z)=\left ( h_{1}\left ( z \right ) ,\cdots , h_{r}\left ( z \right ) \right ).$
 \end{lemma}

\section{Main results}\label{sec3}
Let us  first investigate the K\"ahler potential functions of a  K\"ahler manifold $M$  with a real-analytic K\"ahler metric
$g_{M}$.
Let $\psi$ be a   K\"ahler potential  function on  $ U\subset M$. There exists  a local coordinate system $z$ on
a neighbourhood $U$ of a point $p \in M$,   such that $\psi (z):U\rightarrow \mathbb{R}$
can be analytically continued to an open neighbourhood
$W \subset U \times \text{conj}(U)$,
where $\text{conj}(U)=\{w\in \mathbb{C}^{n}| \overline{w} \in U\}$.
This extension function is called a polarization function of $\psi$ on $U$, and we denote it  by $\psi(z,w)$.
Obviously, it is unique and $\psi(z)=\psi(z, \bar z)$.

 By removing the pure terms that only contain $z$ or $w$, the following lemma describes a usual method  for constructing  a new K\"ahler potential function and its polarization function  from the old one (see also the proof of Lemma 4.1 in \cite{Huang2014}).

\begin{lemma}\label{lemk}
 Let $M$ be a complex manifold  with a real-analytic K\"ahler metric
$g_{M}$.
Denote by  $\psi$ an arbitrary  K\"ahler potential function of $g_M$
in a local coordinate system $(U,z)$ around  $p\in M$.
Then there is an another K\"ahler potential function $\varphi$ such that its
polarization function $\varphi$ satisfying $\varphi(z, \bar z_{0})=0$ (i.e., $(\exp\varphi)(z, \bar z_{0})=1$),
where  $z_0$ denotes   the  coordinate of $p$ and $(\exp\varphi)(z, w)$ is the polarization function of $\exp(\varphi(z))$.

Moreover,
if  the polarization function $\psi(z,  w)$ (resp.  $\phi(z,w):=(\exp \psi)(z, w)$) is holomorphic Nash algebraic  in $(z,w)\in U \times \text{conj}(U)$,
then
\begin{description}
  \item[(i)] The  polarization function $\varphi(z, w)$ (resp. $(\exp \varphi)(z,w)$) is  holomorphic Nash algebraic   in $(z, w)\in U \times \text{conj}(U)$.
  \item[(ii)] The polarization function of $\det(\psi_{i\bar{j}}(z))_{1\leq i, j\leq m}$ is
$$\det(\frac{\partial^{2}\psi(z,w)}{\partial z_{i}\partial {w_{j}}}),$$
and it is a holomorphic Nash algebraic function in $(z, w)\in U \times \text{conj}(U)$.
\end{description}
\end{lemma}
\begin{proof}
Let $\psi(z,w)$ be the polarization function of  $\psi$ on $U \times \text{conj}(U)$.
Define
\begin{equation}\label{2}
\varphi(z,w)
=\psi(z,w)-\psi(z, \bar z_{0})-\psi(z_{0}, w)+\psi(z_{0}, \bar z_{0}).
\end{equation}
Then $\varphi(z)=\varphi(z,\bar{z})$ is a new K\"ahler potential on $U$ satisfying $\varphi(z, \bar z_{0})=0$.
Obviously,
$\varphi(z,w)$ is the polarization function of $\varphi(z)$.
Denote by  $(\exp\psi)(z,w)$  the polarization function of  $\exp(\psi(z))$.
Since $(\exp\varphi)(z,w)=\exp(\varphi(z, w))$,
we know that $\varphi(z, \bar z_{0})=0$ is equivalent to $(\exp\varphi)(z,\bar z_{0})=\exp(\varphi(z, \bar z_{0}))=1$.

\begin{description}
  \item[(i)] According to  the properties of holomorphic Nash algebraic function in Section \ref{sec2},
conclusion (i) follows from  the definition of $\varphi(z,w)$ and the following observation
\begin{eqnarray*}
(\exp \varphi)(z,w)&=&\exp(\varphi(z,w))\\
&=&\frac{\exp(\psi(z,w))\exp(\psi(z_{0}, \bar z_{0}))}{\exp(\psi(z, \bar z_{0}))\exp(\psi(z_{0}, w))}
=\frac{(\exp\psi)(z,w)(\exp\psi)(z_{0}, \bar z_{0})}{(\exp\psi)(z, \bar z_{0})(\exp\psi)(z_{0}, w)}.
 \end{eqnarray*}
  \item[(ii)] For conclusion (ii),
 we have
$\psi_{i\bar{j}}(z)=\frac{\partial^{2}\psi(z, \bar z)}{\partial z_{i}\partial \bar{z}_{j}}
=\frac{\partial^{2}\psi(z,w)}{\partial z_{i}\partial {w_{j}}}|_{w=\bar{z}}$ on  $U \times \text{conj}(U)$.
By the uniqueness, we obtained that the polarization function of $\psi_{i\bar{j}}(z)$ is $\frac{\partial^{2}\psi(z,w)}{\partial z_{i}\partial {w_{j}}}$.
This implies that
the polarization function of $\psi_{i\bar{j}}(z)$ is
  holomorphic Nash algebraic if $\psi(z,w)$ is.

Notice that  $\phi(z,w)=(\exp \psi)(z, w)=\exp (\psi(z, w))$.
Then
$\psi(z)=\log \phi(z,z)$.
Hence, we have
$\psi_{i\bar{j}}(z)=\frac{\partial^{2}\log\phi(z, w)}{\partial z_{i}\partial w_{j}}|_{w=\bar z}$.
 The polarization function of $\psi_{i\bar{j}}(z)$ is
$$\frac{\partial^{2}\log(\phi(z,w))}{\partial z_{i}\partial {w_{j}}}
=-\frac{1}{\phi^{2}(z,w)}\frac{\partial\phi(z,w)}{\partial z_{i}}\frac{\partial\phi(z,w)}{\partial w_{j}}+\frac{1}{\phi(z,w)}\frac{\partial^{2}\phi(z,w)}{\partial z_{i}\partial {w_{j}}}.$$
This also implies that
the polarization function of $\psi_{i\bar{j}}(z)$ is
  holomorphic Nash algebraic if  $\phi(z,w)$  is.
This shows that
  the polarization function of $\det(\psi_{i\bar{j}}(z))$
is the function $
\det(\frac{\partial^{2}\psi(z,w)}{\partial z_{i}\partial {w_{j}}})$ in $ U \times \text{conj}(U)$. And  it is holomorphic Nash algebraic if $\psi(z,w)$ or $\phi(z,w)=(\exp \psi)(z, w)$  is.
\end{description}
\end{proof}

\begin{definition}\label{def EN}
Let $M$ be a K\"ahler manifold  with a real-analytic K\"ahler metric
$g_{M}$.
For any point $p \in M$, if there exists  a local
coordinate neighbourhood $U\subset M$,
and  a K\"ahler potential function $\psi$ of $g_{M}$ in $U\subset M$,
such that the polarization function of $\psi$ (resp. $\exp\psi$) is a holomorphic Nash algebraic function on $U\times \text{conj}(U)$,
we say  $g_{M}$ is Nash-algebraic (exponent Nash-algebraic).
And $(M, g_{M})$ is called   a Nash-algebraic K\"ahler manifold (exponent Nash-algebraic  K\"ahler manifold).
\end{definition}

\begin{lemma}\label{prop1.3}
Let $(M, g_M)$ be a K\"ahler manifold and suppose that $\omega'_{M}:=\alpha\mathrm{Ric}(\omega_{M})+\beta\omega_M$  induces a K\"ahler metric $g'_M$ on $M$.
For any $\mu\in \mathbb{R}^{+}$,  $(M,  \mu g'_M)$ and the complex Euclidean space do not have common K\"ahler submanifolds if $\omega_M$ is exponent Nash-algebraic.
\end{lemma}
\begin{proof}
Suppose that  $(\mathbb{C}^{n}, g _{\mathbb{C}^{n}})$ and   $(M, \mu g'_M)$ have common K\"ahler submanifolds. Then
they have an one dimensional complex submanifold $V\subset \mathbb{C}$,
holomorphic mappings  $F:V\to \mathbb{C}^{n}$ and   $G_{1}:V\to M$
 such that
\begin{equation}\label{0}
 	F^{*} \omega _{\mathbb{C}^{n}  } = \mu G_{1}^* \omega'_{M} ~~on ~~V.
\end{equation}
Without loss of generality, assume that $0\in V$, $F=(f_1, \cdots, f_n): V\rightarrow \mathbb{C}^n$, and $G_1=(g_1, \cdots$ $, g_m): V\rightarrow U\subset M$ are holomorphic mappings, such that $F(0)=0$ and $G_1(0)=0$, where  $(U, z)$ is a local coordinate system containing $G_1(V)$.
By Lemma \ref{lemk}, there always exists a K\"ahler potential function $\varphi$ of    $\omega_M$  on $U$ such that its  polarization function $\varphi$ satisfying $\varphi(z, 0)=0$ (i.e., $\phi(z,0):=(\exp\varphi)(z, 0)=1$).
By \eqref{0}, we have
\begin{align*}
\partial\bar{\partial}(\sum_{j=1}^n|f_j(s)|^2)&=-\mu\alpha\partial \bar\partial\log\det(\varphi_{i\bar{j}}(G_1(s)))+\mu\beta\partial\bar{\partial}\log\phi(G_1(s)).
\end{align*}
Get rid of  $\partial\bar{\partial}$, there exists a holomorphic function $\mathfrak{a}_{0}(s)$  on $V$ such that
\begin{align*}
\sum_{j=1}^n|f_j(s)|^2&=-\mu\alpha\log\det(\varphi_{i\bar{j}}(G_1(s)))+\mu\beta\log\phi(G_1(s))+\mathfrak{a}_{0}(s)+\overline{\mathfrak{a}_{0}(s)}.
\end{align*}
By polarizing, we get
\begin{align*}
\sum_{j=1}^n f_j(s)\bar{f}_j(t)=&-\mu\alpha\log\det(\varphi_{i\bar{j}}(G_1(s),\bar {G}_1(t)))+\mu\beta\log\phi(G_1(s), \bar {G}_1(t))+\mathfrak{a}_{0}(s)+\bar{\mathfrak{a}}_{0}(t),
\end{align*}
where $(s, t)\in V \times \mathrm{conj}(V)$, $\mathrm{conj}(V)=\{s\in\mathbb{C}: \bar{s}\in V\}$,
$ \bar{\mathfrak{a}}_0(t)=\overline{\mathfrak{a}_0(\bar t)}$, $\bar{f}_j(t)=\overline{f_j(\bar{t})} (1\leq j\leq n)$, $\bar{g}_i(t)=\overline{g_i(\bar{t})} (1\leq i\leq m)$.
Define
 $H(z,w)=\det(\varphi_{i\bar{j}}(z,w))_{1\leq i, j\leq m}$ and $\widetilde{H}(z,w)=\frac{H(z,w)H(0,0)}{H(0,w)H(z,0)}$.
Then there exists a holomorphic function $\mathfrak{a}_{1}(s)$ such that
\begin{align*}
\sum_{j=1}^n f_j(s)\bar{f}_j(t)
=-\mu\alpha\log \widetilde{H}(G_1(s),\bar {G}_1(t)) +\mu\beta\log\phi(G_1(s), \bar {G}_1(t))+\mathfrak{a}_{1}(s)+\overline{\mathfrak{a}}_{1}(t).
\end{align*}
Since $\widetilde{H}(z,0)=1$, and  $\phi(z, 0)=1$, we get $\overline{\mathfrak{a}}_{1}(t)+\mathfrak{a}_{1}(0)=0$,
$\mathfrak{a}_{1}(s)+\overline{\mathfrak{a}}_{1}(0)=0$ and
$\mathfrak{a}_{1}(0)+\overline{\mathfrak{a}}_{1}(0)=0$
by taking $s=0$ or $t=0$ or $s=t=0$ respectively.
This implies that $\mathfrak{a}_{1}(s)+\overline{\mathfrak{a}}_{1}(t)=0$ and
\begin{align}\label{a}
\sum_{j=1}^n f_j(s)\bar{f}_j(t)
=&\log \left(\widetilde{H}^{-\mu\alpha}(G_1(s),\bar {G}_1(t)) \phi^{\mu\beta}(G_1(s), \bar {G}_1(t))\right).
\end{align}
Notice that $\widetilde{H}(z,w)$ and $\phi(z,w)$ are Nash algebraic.
By Theorem 2.1 in \cite{Zhang2023}, for $p\in V$, we have
\begin{equation}\label{b}
\sum_{j=1}^n |f_j(s)|^{2} \in \widetilde{\Lambda}_{p}, \ \text{and} \ \ \log\left(\widetilde{H}^{-\mu\alpha}(G_1(s),\bar {G}_1(s)) \phi^{\mu\beta}(G_1(s), \bar {G}_1(s))\right)\notin \widetilde{\Lambda}_{p}\backslash \mathbb{R}_{p},
\end{equation}
where
$\displaystyle\widetilde{\Lambda}_{p}=\left\{a_{0}+\sum_{j=1}^{\kappa}a_{j}|\chi_{j}|^{2}\big| \chi_{j}\in \mathcal{O}_{p}, \chi_{j}(p)=0, a_{j}\in \mathbb{R}, \kappa\in \mathbb{Z}^{+}\right\}$,
$\mathcal{O}_{p}$ denotes the local ring of germs of
holomorphic functions at $p$ and
$\mathbb{R}_{p}$ denotes the germs of real numbers
(See Section 2.1 in \cite{Zhang2023} for the details).
It  is a conflict.
\end{proof}

Before giving the proof of Theorem \ref{thm1}, we recall the definitions of the disc bundle,  the ball bundle  and their  ansatzes.
 Let  $\pi:(L, h)\rightarrow M$ be a Hermitian line bundle over
 a K\"ahler manifold $(M, g_{M})$ of dimension $m$ satisfying  $\sqrt{-1}\partial \bar\partial\log h=l\omega_{M}$, $l\in \mathbb{R}$.
 The disc bundle is given by
 \begin{equation}\label{D1}
 D(L) := \{v \in L : |v|_{h} < 1\}.
 \end{equation}

 For local
coordinates $(U, z)$ of $M$ and a natural local free frame $\{e_{U}\}$ on $U$,  there exists a system of linear holomorphic coordinates $(\xi)$ on each fiber of $\pi^{-1}(U)$,
 so  $(z, \xi)$ is a holomorphic coordinate system of $\pi^{-1}(U)$.
One can locally represent the Hermitian structure on $L$ by a positive function $h(z)$  on $U$ such that its Hermitian form can be expressed as
$|v|_{h}^{2}=h(z)\xi\bar{\xi}$, and denoted by $x$ for brevity.
Let $u$ be a smooth real-valued  function on $[0,+\infty)$.
Then the following  $(1, 1)$-form
\begin{equation}\label{womega1}
\omega_{D}=\alpha\pi^{*}(\mathrm{Ric}(\omega_M))+\beta\pi^{*}(\omega_M)+\sqrt{-1}\partial\bar\partial u(|v|_{h}^{2})
\end{equation}
is well defined on $D(L)$, where
 $\mathrm{Ric}(\omega_M) $ denotes the associated Ricci form
 $-\sqrt{-1}\partial \bar\partial \log \det(g^{M})$.

For any fixed $k\in \mathbb{Z}^{+}$, set
$(E_{k}, H_{k}) = (L, h)\oplus \cdots\oplus (L, h).$
There are $k$ copies of $(L, h)$ on the right hand side.
The ball bundle is defined by
 \begin{equation}\label{Bk}
 \mathbb{B}(E_{k}) := \{v \in E_{k} : ||v||_{H_{k}}^{2} < 1\},
 \end{equation}
where $||v||_{H_{k}}^{2}=\langle\xi,\xi\rangle h(z)=\sum_{j=1}^{k}\xi_{j}\bar\xi_{j}h(z)$.
Define an $(1,1)$-form on $E_{k}$ by
\begin{equation}\label{omega Dk}
\omega_{\mathbb{B}(E_{k})}:=\alpha\pi^{*}\mathrm{Ric}(\omega_{M})+\beta\pi^{*}(\omega_M)+\sqrt{-1}\partial\bar\partial u(||v||_{H_{k}}^{2}).
\end{equation}

\remark If $\alpha=0, \beta=1$, $\sqrt{-1}\partial \bar\partial\log h=l\omega_{M}$, then
$\omega_{\mathbb{B}(E_{k})}$ turns to be the well known Calabi's ansatz \cite{Calabi1979}.
It has been studied by many people   on the  complete  K\"ahler metrics
with various constant curvatures such as Ricci curvature or scalar curvature   \cite{Calabi1979, Hwang, Hao2023}.

\remark If $\alpha=-\frac{1}{m+k+1}, \beta=\frac{1}{m+k+1}$,  $\sqrt{-1}\partial \bar\partial\log h=\omega_{M}$,
then $\omega_{\mathbb{B}(E_{k})}$ is the $(1, 1)$-form introduced by Ebenfelt, Xiao and Xu \cite{Ebenfelt2023} to study  the existence of the complete K\"ahler-Einstein metric on ball bundles.

\begin{proof}[The proof of Theorem 1]
Suppose that there exists a  connected open subset $V\subset \mathbb{C}$,
and  non-constant holomorphic mappings $F:V\to \mathbb{C}^{n}$ and   $G:V\to \mathbb{B}(E_{k})$  such that the K\"ahler forms satisfying
\begin{equation}\label{1}
 	F^{*} \omega _{\mathbb{C}^{n}  } =\mu G^* \omega_{\mathbb{B}(E_{k})} ~~on ~~V
\end{equation}
for a real constant $\mu$.

Let $U\subset M$ be a local
coordinate neighbourhood such that $G(V)\subset \mathbb{B}(E_{k})\cap\pi^{-1}(U)$ (shrink $V$ if needed).
Without loss of generality, assume that $0\in V$ and $F=(f_1, \cdots, f_n): V\rightarrow \mathbb{C}^n$, $G=(G_1,G_{2})=(g_1, \cdots, g_m, g_{m+1},\cdots, g_{m+k}): V\rightarrow \pi^{-1}(U)\cap \mathbb{B}(E_{k})$ are holomorphic mappings, such that $F(0)=0$ and $G(0)=(z_{0}, \xi_0)$, since $\pi^{-1}(U)$ is locally homeomorphic to $U\times\mathbb{C}^{k}$.

Since $(M, g_{M})$ is  exponent Nash-algebraic, there
exists a K\"ahler potential function $\psi$ of $g_{M}$ in $U\subset M$,
such that the polarization function of  $\exp\psi$ is a holomorphic Nash algebraic function on $U\times \text{conj}(U)$.
By Lemma \ref{lemk}, we can replace $\psi$  by a new K\"ahler potential function $\varphi$
such that the  polarization function  $(\exp \varphi)(z,w)$ of $\exp (\varphi(z))$ and the  polarization function of $\det(\varphi_{i\bar{j}}(z))_{1\leq i, j\leq m}$ are  holomorphic Nash algebraic   in $(z, w)\in U \times \text{conj}(U)$
and also satisfies  $\exp\varphi(z, \bar z_{0})=1$.
According the assumption  $\sqrt{-1}\partial \bar\partial\log h=l\omega_{M}$, we have
$h(z)=|\mathrm{hol}|^{2}\exp (l\varphi(z))$.
By suitable choice for the local  free frame of $L$,
we can make $h(z)=\exp (l\varphi(z))$ in $U$ and its   polarization function
$h(z,w)=\exp (l\varphi(z,w))$  in $U \times \text{conj}(U)$.
Furthermore,
the polarization function of $x(z,\xi):=||v||^{2}_{H_{k}}=h(z)\|\xi\|^{2}$ is $\langle\xi, \zeta\rangle\exp (l\varphi(z, w))$.
In particular, we have
$
\langle\xi, \bar{\xi}_{0}\rangle\exp (l\varphi(z,\bar{z}_{0}))=\langle\xi, \bar{\xi}_{0}\rangle.
$

Define $\phi(z):=\exp(\varphi(z))$.
Then the K\"ahler potential $\varphi(z)=\log\phi(z)$ and
we have
$$\omega_{\mathbb{B}(E_{k})} =-\mu\alpha\partial \bar\partial\log\det(\varphi_{i\bar{j}}(z))+\mu\beta\partial \bar\partial\log\phi(z)+\mu\partial \bar\partial u(x(z,\xi)),$$
where $x(z,\xi)=\|\xi\|^{2}\phi^{l}(z).$
By the proof of Lemma \ref{prop1.3}, we know $G(V)\not\subset M$, i.e., $G_{2}\neq 0$.
By \eqref{1},   for any $s\in V$, we have
\begin{align*}
\partial\bar{\partial}(\sum_{j=1}^n|f_j(s)|^2)&=-\mu\alpha\partial \bar\partial\log\det(\varphi_{i\bar{j}}(G_1(s)))+\mu\beta\partial\bar{\partial}\log\phi(G_1(s))
+\mu\partial\bar{\partial}u(x(G(s))),
\end{align*}
where $x (G(s))=\phi^{l}(G_{1}(s))\|G_{2}(s)\|^{2}$.
Get rid of  $\partial\bar{\partial}$, there exists a holomorphic function $\mathfrak{a}(s)$  on $V$ such that
\begin{align*}
\sum_{j=1}^n|f_j(s)|^2=&-\mu\alpha\log\det(\varphi_{i\bar{j}}(G_1(s)))+\mu\beta\log\phi(G_1(s))\\
&+\mu u(x (G(s)))+\mathfrak{a}(s)+\overline{\mathfrak{a}(s)}~~\text{for}~~ s\in V.
\end{align*}
By polarizing, we get
\begin{align*}
\sum_{j=1}^n f_j(s)\bar{f}_j(t)=&-\mu\alpha\log\det(\varphi_{i\bar{j}}(G_1(s),\bar {G}_1(t)))_{1\leq i, j\leq m}+\mu\beta\log\phi(G_1(s), \bar {G}_1(t))\\
& +\mu u(x (G(s), \bar G(t)))+\mathfrak{a}(s)+\overline{\mathfrak{a}}(t),
\end{align*}
where $(s, t)\in V \times \mathrm{conj}(V)$, $\mathrm{conj}(V)=\{s\in\mathbb{C}: \bar{s}\in V\}$,
$ \overline{\mathfrak{a}}(t)=\overline{\mathfrak{a}(\bar t)}$, $\bar{f}_j(t)=\overline{f_j(\bar{t})} (1\leq j\leq n)$, $\bar{g}_i(t)=\overline{g_i(\bar{t})} (1\leq i\leq m+1)$,  and
$\bar G(t)=(\bar{G}_1(t), \bar{G}_{2}(t))= (\bar{g}_1(t), \cdots, \bar{g}_m(t), \bar{g}_{m+1}(t),\cdots,\bar{g}_{m+k}(t))$,
$x (G(s), \bar G(t))$ $=\phi^{l}(G_{1}(s), \bar G_{1}(t))<G_{2}(s), G_{2}(t)>$.

Consider the set $S=\{f_{1},\cdots,f_{n},g_{1},\cdots,g_{m+1},\cdots,g_{m+k}\}$.
By Lemma \ref{field}, there exist Nash functions $\widehat{f}_{1},\cdots,\widehat{f}_{n}$, $\widehat{g}_{1},\cdots,\widehat{g}_{m+1},\cdots,\widehat{g}_{m+k}$ on
a neighbourhood $\widehat{V}$ in $\mathbb{C}^{n}\times \mathbb{C}^{r}$
such that $f_{j}(z)=\widehat{f}_{j}(z,h_{1}(z),\cdots,$ $h_{r}(z))$ and
$g_{j}(z)=\widehat{g}_{j}(z,h_{1}(z),\cdots, h_{r}(z))$, where
$\{h_{1},\cdots, h_{r}\}$ is a maximal algebraic independent subset over $\Re$.

Define $\widehat{G}(s,X):=(\widehat{G}_1(s,X),\hat{G}_{2}(s, X))$=$(\hat{g}_1(s, X), \cdots, \hat{g}_m(s, X),\hat{g}_{m+1}(s, X),\cdots,\hat{g}_{m+k}(s, X))$
on $\widehat{V}$, and
\begin{align}\label{Psi1}
\Psi(s, X, t)\nonumber
:=&\sum_{j=1}^n\hat{f}_j(s, X)\bar{f}_j(t)+\mu\alpha\log\det(\varphi_{i\bar{j}}(\widehat{G}_1((s, X),\bar {G}_1(t))))
\nonumber\\
&-\mu\beta\log\phi(\widehat{G}_1(s,X), \bar{G}_1(t))-\mu \log[\exp (u(x (\widehat{G}(s,X), \bar G(t))))]-\mathfrak{a}(s)-\overline{\mathfrak{a}}(t),
\end{align}
and
\[\Psi_\tau(s, X, t):=\frac{\partial^\tau\Psi}{\partial t^\tau}(s, X, t),\]
where $(s, X, t)\in \widehat{V}\times\overline{V}$, $X=(X_1, \cdots, X_r)$, $\tau\ge1$.

We now prove that   $\Psi(s, X, t)=\Psi(s, X, 0)$ for any $t$ near $0$  and any $(s, X)\in\widehat{V}$.
Actually, if it is not true, then
$\Psi(s, X, t)=\Psi(s, X, 0)+\sum_{\tau\ge1}\Psi_\tau(s, X, 0)t^\tau$
in a neighbourhood of $t=0$ since  $\Psi(s, X, t)$ is complex analysis in $t$.

Notice that $\hat{f}_1(s, X), \cdots, \hat{f}_n(s, X)$, $\hat{g}_1(s, X), \cdots, \hat{g}_m(s, X),\cdots, \hat{g}_{m+k}(s, X)$ are all Nash algebraic
functions in $(s, X)\in\widehat{V}$,
some tedious manipulation yields $\Psi_\tau(s, X, 0)$ is also Nash algebraic in $(s, X)$. Suppose that $\Psi_\tau(s, X, 0)$ is not constant, there exists $P(s, X, y)=A_d(s, X)y^d+\cdots+A_0(s, X)$, $A_0(s, X)\neq0$, such that $P(s, X, \Psi_\tau(s, X, 0))\equiv0$. For any $s\in V$, the equation $\Psi(s, h_1(s), \cdots, h_r(s), t)\equiv0$ implies that $\Psi_\tau(s, h_1(s), \cdots, h_r(s), 0)\equiv0$. Hence, we can obtain that $A_0(s, h_1(s), \cdots, h_r(s))\equiv0$. It contradicts with  $\{h_1(s), \cdots, h_r(s)\}$ is a maximal algebraic independent set in $\Re$.
 Thus $$\Psi_\tau(s, X, 0)=\Psi_\tau(s, h_1(s), \cdots, h_r(s), 0)\equiv0.$$
By \eqref{Psi1}, we get
\begin{align}\label{psi2}
\Psi(s, X, t)
=&\Psi(s, X, 0)=\sum_{j=1}^n\hat{f}_j(s, X)\bar{f}_j(0)+\mu\alpha\log\det(\varphi_{i\bar{j}}(\widehat{G}_1((s, X),\bar {G}_1(0))))_{1\leq i, j\leq m}\nonumber\\
&
-\mu\beta\log\phi(\widehat{G}_1(s,X), \bar{G}_1(0))-\mu\log (\exp u(x (\widehat{G}(s,X), \bar G(0))))-\mathfrak{a}(s)-\overline{\mathfrak{a}}(0),
\end{align}
where $\bar{G}_1(0)$ and $\bar G(0)$ are denoted by the point $\bar{z}_{0}$ and $(\bar{z}_{0}, \bar{\xi}_{0})$, respectively.

We claim that
 there exists $(s_0, t_0) \in V\times \overline{V}$ with $t_0\ne 0$ such that
$\sum_{j=1}^n\hat{f}_j(s_0, X)\bar{f}_j(t_0)$ is not constant in $X$. In fact,
if it is a constant in  $V\times \overline{V}$, let $t=\bar{s}$, $X=(h_1(s), \cdots, h_r(s))$, then
\[\sum_{j=1}^n|f_j(s)|^2=\sum_{j=1}^n\hat{f}_j(s, h_1(s), \cdots, h_r(s))\bar{f}_j(\bar{s})\equiv constant, ~~~ s\in V.\]
 It is impossible.

Choosing the points $s_0, t_0$ as described above, then $\sum_{j=1}^n\hat{f}_j(s_0, X)\bar{f}_j(t_0)$ is a holomorphic Nash algebraic function in  $X$.
Combine \eqref{Psi1} with \eqref{psi2}, we get
\begin{align*}
&\exp\left(\frac{1}{\mu}\displaystyle{\sum_{j=1}^n\hat{f}_j(s_0, X)\bar{f}_j(t_0)-\frac{1}{\mu}\overline{\mathfrak{a}}(t_0)+\frac{1}{\mu}\overline{\mathfrak{a}}(0)}\right)\\
&=\frac{\det^{\alpha}\left(\varphi_{i\bar{j}}(\widehat{G}_1((s_0, X),\bar {G}_1(0)))\right)\phi^{\beta} (\widehat{G}_1((s_0, X), \bar{G}_{1}(t_{0}))) \exp u(x (\widehat{G}(s_{0},X), \bar G(t_{0})))}
{\det^{\alpha}\left(\varphi_{i\bar{j}}(\widehat{G}_1((s_0, X),\bar {G}_1(t_{0})))\right)\phi^{\beta} (\widehat{G}_1((s_0, X), \bar{G}_{1}(0))) \exp u(x (\widehat{G}(s_{0},X), \bar G(0)))},
\end{align*}
where $\phi(z, w)=\exp(\varphi(z, w))$ and $x ((z, \xi), (w, \zeta))=\langle\xi, \zeta\rangle h(z,w)=\langle\xi, \zeta\rangle(\exp(\varphi(z, w)))^{l}$.
By  Lemma \ref{Ma2024} and Lemma \ref{lemk},  $F$ must be a constant map. We complete the proof.
\end{proof}

\section{Some examples}\label{sec:4}
\subsection{Three kinds of  ball bundles equipped with their Bergman metrics}\label{sec4.2}
Let $M$ be a complex manifold and denote by $L$ its canonical line
bundle. Suppose the Hilbert space $H^2(M, L)$ of square-integrable holomorphic $m$-forms
on $M$ has no base points, and denote by $\mathcal{K}_M(z, w)$ the Bergman kernel form on $L$. Regarding
$\mathcal{K}_M(z, z')= K_M(z, z')dz_1 \wedge \cdots \wedge dz_m \wedge \overline{dz'_1} \wedge \cdots \wedge \overline{dz'_m}$ as a Hermitian metric $h^{*}$ on the anti-canonical line bundle $L^{*}$. Denote by
$$\Theta_{h}=-\sqrt{-1} \partial\overline{\partial}\log K_M(z, z')^{-1} \geq 0$$ the curvature form of the dual metric $h$ on $L$, and write
$ds_{M}^{2}
$
for the corresponding semi-K\"ahler metric on $M$. We say that $ds_{M}^{2}$ is a \textbf{Bergman metric} whenever $ds_{M}^{2}$
is positive definite.
If the manifold is just a domain in $\mathbb{C}^{m}$, then by the identification
\begin{equation}\label{n0form}
f(z) dz_{1} \wedge\cdots \wedge dz_{m} \rightarrow f(z)
\end{equation}
of $(m, 0)$-forms with functions, one recovers the usual definition of the Bergman space and
Bergman kernel of domains in $\mathbb{C}^{m}$. The function $K_M(z, z')$ associated with the form $\mathcal{K}_M(z, z')$ is the usual Bergman kernel function of the domains.\\

Let $G^{*}$ be the holomorphic automorphism group, $K$ be an isotropy  subgroup at some point of $M$.
Then the \textbf{compact Hermitian symmetric space} $M=G^{*}/K$. Let $L$ be the homogeneous line bundle over $M = G^{*}/K$ induced by the representation
$k \rightarrow (\det(\mathrm{Adk}))^{\frac{1}{p}}, k \in K$. Denote by $p$ the genus of $G^{*}/K$. The bundle
$L^{p}$ is then the top exterior product $\bigwedge^{m} T(1,0)$ of the holomorphic tangent bundle over  $G^{*}/K$.
Using the local coordinates $\mathfrak{p}^{+}:z \rightarrow \exp(iz) \in G^{*}/K$, we have that the fiber metric in $L^{p}$
is given by
$$\|\partial_{1}\wedge\cdots \wedge \partial_{m}\|_{z}= h(z, -z)^{-p}.$$
Denoting by $e(z)$ a local holomorphic section of $L$ so that $e(z)^{p} = \partial_{1}\wedge\cdots \wedge \partial_{m}$, we see
that the metric on $L$ is given by
$$\|e(z)\|^{2}_{z} = h(z, -z)^{-1},$$
where $h(z, -z)$ is an irreducible polynomial  called the Jordan canonical polynomial.
Let $L^{*}$ be the dual bundle of $L$ and  $e^{*}(z)$ be the local section of $L^{*}$ dual to $e(z)$.
Let $D(L^{*\mu})$ and  $S(L^{*\mu})=\partial D(L^{*\mu})$ be the unit disc bundle and  the unit circle bundle of  the higher powers  $L^{*\mu}$ of
$L^{*}$, $\mu = 0, 1, 2, \cdots$, i.e.,
 $$D(L^{*\mu}) = \{\xi\in L^{*\mu}: \|\xi\|^{2}< 1\}.$$
A local defining function for $S(L^{*\mu})$ is
given by
$$\rho(z, \lambda e^{*\mu}(z)) = |\lambda|^{2}h(z, -z) - 1, \lambda \in \mathbb{C}, z \in  \mathfrak{p}^{+},$$
where $e^{*\mu}(z)=e^{*}(z)\otimes\cdots \otimes e^{*}(z)$ is the local section of $L^{*\mu}$.

Identifying $\mathfrak{p}^{+}$ with a dense open subset of $M$ of full measure as described in \cite{Englis2010}, and using the local trivializing section $e^{*}(z)$ as before, the correspondence
$\mathfrak{p}^{+} \times \mathbb{C} \ni(z, \lambda) \mapsto \xi = (z, \lambda e^{*\mu}(z)) \in L^{*\mu} $ sets up a bijection between a dense open subset
of $D$ of full measure and the Hartogs domain
$$\Omega_{\mu}= \{(z, \lambda) \in \mathfrak{p}^{+} \times \mathbb{C} : |\lambda|^{2}h^{\mu}(z, -z) < 1\}.$$

Denote by
$$\rho(x, \alpha; y,\beta) = \alpha \bar\beta h(x, -y)^{\mu} - 1,$$
the sesqui-holomorphic extension of the defining function $\rho$.
\begin{lemma}\label{Englis} \cite{Englis2010}
The Bergman kernel of the disc bundle $D(L^{*\mu})$ is given, in local coordinates
$\alpha e^{*\mu}(z)\rightarrow (z, \alpha) \in \mathfrak{p}^{+} \times \mathbb{C}, |\alpha|^{2}h(z, -z)^{\mu} < 1$, by
$$K_{D}(x, \alpha; y, \beta)=K^{*}(x, \alpha; y,\beta) dx_{1} \wedge\cdots\wedge dx_{m} \wedge d\alpha \wedge d\bar{y}_{1}\wedge\cdots \wedge d\bar{y}_{m} \wedge d\bar{\beta},$$
where
$$K^{*}(x, \alpha; y, \beta)=\sum_{\nu=0}^{+\infty} \frac{\nu+1}{\pi}c_{\mu,\nu}h(x, -y)^{\mu \nu+1}(\alpha \bar\beta)^{\nu}.$$
where $c_{\mu, \nu}=\frac{((\mu \nu+p-\frac{m}{r}))\frac{m}{r}}{((p-\frac{m}{r} ))\frac{m}{r}}$.
It has an expansion in terms of the sesqui-holomorphically extended defining function
$\rho(x, \alpha; y,\beta)=\alpha\bar\beta h(x,-y)^{\mu}-1$ as
$$\frac{K^{*}(x, \alpha; y,\beta)}{h(x, -y)} = c_{0}\rho(x, \alpha; y,\beta)^{-m-2} +\cdots+ c_{m+1}\rho(x,\alpha; y,\beta)^{-1},$$
where $c_0 = (-1)^{m+2} \frac{(m + 1)!\mu^m}{((p-\frac{m}{r} ))_{\frac{m}{r}}}$,
$c_j$ are some real constants and
$$\rho(x, \alpha; y,\beta) = \alpha \bar\beta h(x, -y)^{\mu} - 1.$$
\end{lemma}

By the inflation principle given by Boas, Fu and Straube \cite{Boas1999}, it can be extended to
the unit ball bundle of the direct sum of $L^{*s}$.
For any fix $k\in \mathbb{Z}^{+}$, set
$(E_{k}, H_{k}) = (L^{*\mu}, h^{\mu})\oplus \cdots\oplus (L^{*\mu}, h^{\mu}),$ $k$ copies of $(L^{*\mu}, h^{\mu})$ on the right hand side.
The ball bundle is defined by
 \begin{equation}\label{BkCHSS}
 \mathbb{B}(E_{k}) := \{v \in E_{k} : ||v||_{H_{k}}^{2} < 1\},
 \end{equation}
where $\|v\|_{H_{k}}^{2}=\langle\xi,\xi\rangle h(z,-z)^{\mu}=\sum_{j=1}^{k}\xi_{j}\bar\xi_{j}h(z,-z)^{\mu}$.
The unit circle bundles $S(E_{k}) = \partial  \mathbb{B}(E_{k}) $.
A local defining function for $S(E_{k})$ is
given by
$$\rho(z, \nu) = \|\lambda\|^{2}h(z, -z)^{\mu} - 1, \lambda \in \mathbb{C}^{k}, z \in  \mathfrak{p}^{+},$$
where $\nu=\sum_{j=1}^{k}\lambda_{j} \delta_{j}(z)$, $\lambda=(\lambda_{1}, \cdots,\lambda_{k})$ and $\{\delta_{1}(z),\cdots,\delta_{k}(z)\}$ is the  local frame of $E_{k}$.
\begin{lemma}\label{lem3}
The Bergman kernel of the ball bundle  $\mathbb{B}(E_{k})$ in \eqref{BkCHSS}\ is given, in local coordinates
$\alpha e^{*}(z)\rightarrow (z, \alpha) \in \mathfrak{p}^{+} \times \mathbb{C}^{k}, \|\alpha\|^{2}h(z, -z)^{\mu} < 1$, by
$$K_{\mathbb{B}(E_{k})}(x, \alpha; y, \beta)=K^{*}(x, \alpha; y,\beta) dx_{1} \wedge\cdots\wedge dx_{m} \wedge d\alpha_{1} \wedge\cdots\wedge d\alpha_{k} \wedge d\bar{y}_{1}\wedge\cdots \wedge d\bar{y}_{m} \wedge d\bar{\beta}_{1} \wedge\cdots\wedge d\bar{\beta}_{k},$$
where
\begin{equation}\label{K1}
K^{*}(x, \alpha; y, \beta)=\sum_{j=0}^{+\infty}\frac{(j+1)_{k}}{\pi^{k}}\frac{((\mu(j+k-1)+p-\frac{m}{r}))\frac{m}{r}}{((p-\frac{m}{r} ))\frac{m}{r}} h(x, -y)^{\mu (j+k-1)+1}\langle\alpha, \bar\beta\rangle^{j}.
\end{equation}
It has an expansion in terms of the sesqui-holomorphically extended defining function
$\rho(x, \alpha; y,\beta)=\langle\alpha, \bar\beta\rangle h(x,-y)^{\mu}-1$ as
\begin{equation}\label{K2}
\frac{K^{*}(x, \alpha; y,\beta)}{h^{\mu(k-1)+1}(x, -y)} = C_{0}\rho(x, \alpha; y,\beta)^{-(m+k+1)} +\cdots+ C_{m+1}\rho(x, \alpha; y,\beta)^{-k},
\end{equation}
where  $C_j$ are some real constants.
\end{lemma}
\begin{proof}
Identify the $(m, 0)$-form with functions via \eqref{n0form} in the local coordinates,
and thus the Bergman space on $\mathbb{B}(E_{k})$ can be identified with the space of all functions holomorphic and square-integrable on $\Omega_{\mu}$, i.e. with the usual Bergman space on the Hartogs domain $\mathbb{C}^{m+k}$.
Define $$L_{1}(x,y,t)=\sum_{\nu=0}^{+\infty}\frac{(\nu+1)}{\pi} c_{\mu,\nu}h(x, -y)^{\mu \nu+1}t^{\nu},$$
and $$L_{2}(x,y,t)= c_{0}\sigma(x,  y,t)^{-m-2} +\cdots+ c_{n+1}\sigma(x,  y,t)^{-1},$$
where $\sigma(x,  y,t)=h(x,-y)^{\mu}t-1$.
By the inflation principle and Lemma \ref{Englis}, the Bergman kernel of $B(E_k)$ is
\begin{eqnarray*}
K^{*}(x, \alpha; y,\beta)&=&\frac{1}{\pi^{k-1}}\frac{\partial^{k-1} L_{1}(x,y,t)}{ \partial t^{k-1}}|_{t=\langle\alpha,\beta\rangle}\\
&=&\sum_{j=0}^{+\infty}\frac{(j+1)_{k}}{\pi^{k}} c_{\mu,j+k-1}h(x, -y)^{\mu(j+k-1)+1}t^{j}|_{t=\langle\alpha,\beta\rangle},
\end{eqnarray*}
where $(j + 1)_{k}$ denotes the Pochhammer polynomial of degree $k$, i.e. $(j + 1)_{k} = \frac{\Gamma(j+1+k)}{\Gamma(j+1)},$
and
\begin{eqnarray*}
&&\frac{K^{*}(x, \alpha; y,\beta)}{h(x, -y)}=\frac{1}{\pi^{k-1}}\frac{\partial^{k-1} L_{2}(x,y,t)}{ \partial t^{k-1}}|_{t=\langle\alpha,\beta\rangle}\\
&=& \frac{1}{\pi^{k-1}}h^{\mu(k-1)}(x, -y)\left(\frac{(m+k)!}{(m+1)!}c_{0}\sigma(x,  y,t)^{-(m+k+1)} +\cdots+ (k-1)!c_{m+1}\sigma(x,  y,t)^{-k}\right)|_{t=\langle\alpha,\beta\rangle}.
\end{eqnarray*}
The conclusion  follows.
\end{proof}
From \eqref{K2},  the Bergman metric  of $B(E_k)$ in \eqref{BkCHSS} is given by
\[\omega_{B_{E(k)}}=(\mu(k-1)+1)\pi^*\omega_{M}+\sqrt{-1}\partial\bar{\partial}u( ||\alpha||^2h(x,-x)^{\mu}-1),\]
  where $\omega_{M}=\sqrt{-1}\partial \bar\partial\log h(z, -z)$, and
  $$\exp u(\rho)=  C_{0}\rho(x, \alpha; y,\beta)^{-(m+k+1)} +\cdots+ C_{m+1}\rho(x, \alpha; y,\beta)^{-k}.$$ Notice that $h(z, -w)$ is  the Jordan canonical polynomial of the compact Hermitian symmetric space.
 Thus the base space is an exponent Nash-algebraic K\"ahler manifold.
Corollary  \ref{cor1} follows from Lemma   \ref{lem3} and Theorem \ref{thm1}.

\begin{corollary}\label{cor1}
 Let $(L, h)$ be a Hermitian line bundle such that the $p$-power bundle $L^{p}$ is the top exterior product $\bigwedge^{m} T(1,0)$ of the holomorphic tangent bundle over a compact Hermitian symmetric space $M$.
Let  $(L^{*}, h^{-1})$  be the dual bundle of $(L, h)$.
For any fix $k\in \mathbb{Z}^{+}$, set
$(E_{k}, H_{k}) = (L^{*\mu}, h^{-\mu})\oplus \cdots\oplus (L^{*\mu}, h^{-\mu}).$
There don't exist  common K\"ahler submanifolds between the  complex Euclidean space  and  the ball bundle of $(E_{k}, H_{k})$ endowed with the Bergman metric.
\end{corollary}

Let $\Omega$ be a bounded domain in $\mathbb{C}^{m}$. Define $L=\Omega\times\mathbb{C}$ and  $h$ be the Hermitian metric on $L$ defined by
$h(z,z'):=K^{-1}(z,z)$ for any $(z,\xi)$,$(z',\xi')\in  L$.
Then $\omega_{M}:=\sqrt{-1}\partial \bar\partial \log K(z,z)$ induces the Bergman metric $g_{\Omega}$ on $\Omega$.
Notice that $(E_{k}, H_{k})=(L, h)\oplus \cdots\oplus (L, h)=(\Omega\times\mathbb{C}^{k}, K^{-1}(z,z)I_{k}).$
The ball bundle  is a Hartogs domain over a bounded homogeneous domain. That is
\begin{equation}\label{BH}
\widehat{\Omega}_{m,s}:=\left\{(z,\xi) \in \Omega \times \mathbb{C}^{k}: \|\xi\|^{2}K_\Omega(z,z)^{s} < 1\right\}
\end{equation}
for $s\in \mathbb{R}^{+}$. It is called \textbf{Bergman-Hartogs domain}.
\begin{lemma}\cite{Ishi2017}\label{lem4} If the bounded domain $\Omega$  is homogeneous, then the   Bergman kernel  of Hartogs domain $\widehat{\Omega}_{m,s}$ is
\begin{equation}\label{BK}
K_{\widehat{\Omega}_{m,s}}((z,\xi),(z',\xi'))
=\frac{K_{\Omega}(z,z')^{ms+1}}{\pi^{m}}\sum_{j=0}^{m+1}\frac{c(s,j)(j+m)!}{(1-x)^{j+m+1}}\mid_{x=K_{\Omega}(z,z')^{s}<\xi,\xi'>},
\end{equation}
where $(z,\xi), (z',\xi') \in \widehat{\Omega}_{m,s}$, and $c(s, j)$ are constants depending
on $\Omega$.
\end{lemma}

This lemma shows that the K\"ahler form of the complete Bergman metric on $\widehat{\Omega}_{m,s}$ is
$$
\omega_{\widehat{\Omega}_{m,s}}=(ms+1)\pi^{*}(\omega_{\Omega})-\sqrt{-1}\partial\bar\partial u(K_{\Omega}(z,z)^{s}<\xi,\xi>),
$$
where $\omega_{\Omega}=\partial \bar\partial\log K_{\Omega}(z,z)$ and $\exp u(x)=\sum_{j=0}^{m+1}\frac{c(s,j)(j+m)!}{(1-x)^{j+m+1}}$.

\begin{corollary}\label{cor2}
There don't exist  common K\"ahler submanifolds between the  complex Euclidean space  and the Bergman-Hartogs domain \eqref{BH}  equipped with its Bergman metric $g_{\widehat{\Omega}_{m,s}}$. In particular, it is true for Hartogs domains over bounded symmetric domains in \eqref{BH}.
\end{corollary}
\begin{proof}
It suffices to prove that $(\Omega, g_{\Omega})$ is  a Nash algebraic K\"ahler manifold.
  In general, the Bergman kernel of a bounded homogeneous domain may  not be rational. But it can be proved that,
for any point $p \in \Omega$, there exists  a local
coordinate neighbourhood $U\subset \Omega$,
and  a K\"ahler potential function $\psi$ of the Bergman metric in $U\subset \Omega$,
such that the polarization function of $\psi$ (resp. $\exp\psi$) is a rational function on $U\times \text{conj}(U)$.
In fact, any bounded homogeneous domain $\Omega$  can be equivalent to a Siegel domain $V$ of the second type by a biholomorphic mapping $\phi$. There also exists a rational mapping $\psi$ such that any Siegel domain $V$ can be equivalent to a bounded homogeneous domain.
Define $w:=(\psi\circ\phi(z))$.
Since  the  Bergman kernel of $V$ is rational, the Bergman kernel function of $\psi\circ\phi(\Omega)$
$$K_{(\psi\circ\phi)(\Omega)}(w,w)=K_{\Omega}(z,z)|\det J(\psi\circ\phi)|^{2}=K_{V}(s,s)|\det J\psi|^{2}$$
 is rational and the K\"ahler form can be expressed by $\omega_{\Omega}=\partial \bar\partial\log K_{\psi\circ\phi(\Omega)}(w,w)$ under the new coordinates. Moreover, we get
\begin{equation}
\widehat{\Omega}_{m,s}\cong\left\{(w,\zeta) \in (\psi\circ\phi)(\Omega) \times \mathbb{C}^{k}: \|\zeta\|^{2}K_{(\psi\circ\phi)(\Omega)}(w,w)^{s} < 1\right\}.
\end{equation}
 Hence, we complete the proof.
\end{proof}

Let $L=\{(z, \xi)\in\mathbb{C}^{m}\times\mathbb{C}\}$
 be  the trivial line bundle over  $\mathbb{C}^{m}$.
Let $h$ be the Hermitian metric on $L$ defined by
$h(z,z'):=e^{<z,z'>}$ for any $(z,\xi)$,$(z',\xi')\in  L$.
Then $\omega:= -c(L, h) = \sqrt{-1}\partial \bar\partial \log h=\sqrt{-1}\partial \bar\partial  ||z||^{2}$ induces the flat metric $g_{\mathbb{C}^{m}}$ on $\mathbb{C}^{m}$. Hence, the eigenvalues $\lambda_{i}(z)\equiv0$.
Let $(E_{k}, H_{k})=(L, h)\oplus \cdots\oplus (L, h)=(\mathbb{C}^{m}\times\mathbb{C}^{k},e^{<z,z'>}\mathrm{I}_{k}).$
The ball bundle is
\begin{equation}\label{FBK}
D_{m,k}:=
 \{(z, \xi)\in\mathbb{C}^{m}\times\mathbb{C}^{k}: \|\xi\|^2<e^{-||z||^2}\}.
\end{equation}
It is an unbounded non-hyperbolic strongly pseudoconvex domain called \textbf{Fock-Bergmann-Hartogs domain}.

\begin{lemma} \cite{Yamamori}\label{lem5}
The explicit formula of Bergman kernel on $D_{m,k}$ is
\begin{equation}\label{K_D}
K_{D_{m,k}}((z, \xi), (z', \xi'))=\pi^{-(m+1)}e^{<z, z'>}A_{m}(t)(1-t)^{-(m+2)}\big|_{t=e^{<z, z'>}<\xi, \xi'>},
\end{equation}
where the polynomial $A_{m}(t)=\sum_{j=0}^{m}(-1)^{m+j}(2)_jS(1+m, 1+j)(1-t)^{m-j}$, $(x)_j$ and $S(\cdot, \cdot)$ indicate the Pochammer symbol and the Stirling number of the second kind respectively.
\end{lemma}
\begin{corollary}\label{cor3}
The Fock-Bergmann-Hartogs domain  $D_{m,k}$ equipped with its  Bergman metric  and the complex Euclidean space   have common K\"ahler submanifolds.
\end{corollary}
\begin{proof}
By Lemma \ref{lem5},  the K\"ahler form of the complete Bergman metric  on $D_{m,k}$ is
$$
\omega_{D_{m,k}}=\pi^{*}(\omega_{\mathbb{C}^{m}})-\sqrt{-1}\partial\bar\partial u(\|\xi\|^2e^{\|z\|^2}),
$$
 where $\omega_{\mathbb{C}^{m}}=\sqrt{-1}\partial \bar \partial\|z\|^{2}$ is the K\"ahler form of the complex flat metric, $\exp u(x)=\frac{A_{m}(x)}{(1-x)^{m+2}}$.
Let $f:D\rightarrow D_{m,k}$ be the inclusion mapping, then  $f^{*}(\omega_{D_{m,k}})=\omega_{\mathbb{C}^{m}}$.
Hence, we obtain  the conclusion.
\end{proof}
\subsection{A ball bundle equipped with its  complete K\"ahler-Einstein metric}\label{sec4.1}
Let $(M, g_{M})$ be an $m$-dimensional K\"ahler manifold.
The Ricci tensor of  the K\"ahler metric  $g_{M}$
naturally induces  the Ricci endomorphism of the holomorphic tangent space
$T_p^{1,0}M$
given by $\mathrm{Ric} \cdot (g_{M})^{-1}$ for $p \in M$. The eigenvalues of this endomorphism are called  the Ricci
eigenvalues of $(M, g_{M})$ and are denoted by $\lambda_{1}(p) \leq \cdots \leq \lambda_{m}(p)$ for a fixed $p \in M$. All Ricci eigenvalues are real-valued as both $\mathrm{Ric}$
and $g_{M}$ are Hermitian tensors. By Theorem 1.3, Corollary 1.4 and Remark 5.6 in \cite{Ebenfelt2023} , we know the following result.

\begin{lemma}\cite{Ebenfelt2023}\label{Ebenfelt2023}
The ball bundle $\mathbb{B}(E_{k}) = \{v \in E_{k} : ||v||_{H_{k}} < 1\}$ in \eqref{Bk} admits a unique complete K\"ahler-Einstein metric $g_{\mathbb{B}(E_{k})}$ with Ricci curvature
$-(m+k+1)$ if $(M, g_{M})$ has constant Ricci eigenvalues and every Ricci eigenvalue is strictly less than one. Moreover, this metric is induced by the following K\"ahler form:
$$\omega_{\mathbb{B}(E_{k})}= -\frac{1}{m+k+1}\pi^{*}(\mathrm{Ric}_{M}) +
\frac{1}{m+k+1}\pi^{*}(\omega_{M}) - \sqrt{-1}\partial \bar\partial \log \psi(||v||^2_{H_{k}}),$$
where $\omega_{M}$ and $\mathrm{Ric}_{M}$ are  the K\"ahler form and the Ricci form of $(M, g_{M})$ respectively,
$\psi: (-1, 1) \rightarrow \mathbb{R}^{+}$ is an even real analytic function that depends only on the
characteristic polynomial of the Ricci endomorphism.

In particular, if $g_{M}$ is a complete K\"ahler-Einstein metric with Ricci curvature $-(m+1)$, then
$$\omega_{\mathbb{B}(E_{k})}= -\frac{1}{m+k+1}\pi^{*}(\mathrm{Ric}_{M}) +
\frac{1}{m+k+1}\pi^{*}(\omega_{M}) - \sqrt{-1}\partial \bar\partial \log (1-||v||_{H_{k}}^{2})$$
can generate the complete K\"ahler-Einstein metric on $B(E_k)$.
\end{lemma}

 From Lemma  \ref{Ebenfelt2023} and Theorem \ref{thm1}, we get Corollary \ref{cor4}.
\begin{corollary}\label{cor4}
Let $(M, \omega_{M})$ be a complete K\"ahler-Einstein manifold with Ricci curvature $-(m+1)$.
Let  $\pi:(L, h)\rightarrow M$ be a negative Hermitian line bundle over
 $M$ satisfying  $\sqrt{-1}\partial \bar\partial\log h=l\omega_{M}, l\in \mathbb{R}$.
 Let $E_{k}$ be the direct sum of  $(L, h)$.
If the K\"ahler manifold $(M, \omega_{M})$ is  exponent Nash-algebraic,
then the complex Euclidean space and  $\mathbb{B}(E_{k})$ equipped with its K\"ahler-Einstein metric
do not have common K\"ahler submanifolds.
\end{corollary}



\end{document}